\documentclass[12pt,a4paper]{article}
\usepackage[T1]{fontenc}
\usepackage[utf8]{inputenc}
\usepackage{lmodern}
\usepackage{amsmath,amsthm,amssymb}
\usepackage[margin=25mm]{geometry}
\usepackage{bm}
\usepackage{microtype}
\usepackage{needspace}
\usepackage[hidelinks]{hyperref}

\theoremstyle{definition}
\newtheorem{definition}{Definition}
\theoremstyle{plain}
\newtheorem{thm}[definition]{Theorem}
\newtheorem{prop}[definition]{Proposition}
\newtheorem{lem}[definition]{Lemma}
\newtheorem{cor}[definition]{Corollary}
\theoremstyle{remark}

\DeclareMathOperator{\Ker}{Ker}
\DeclareMathOperator{\Hilb}{Hilb}
\DeclareMathOperator{\Frac}{Frac}
\DeclareMathOperator{\Span}{span}
\DeclareMathOperator{\Gp}{gp}

\hypersetup{
  pdftitle={Generating Sets for the Kernels of Frobenius Polarization Operators},
  pdfauthor={Naoya Yamaguchi}
}
\begin{document}

\title{Generating Sets for the Kernels of\texorpdfstring{\\}{ }Frobenius Polarization Operators}
\author{Naoya Yamaguchi}
\date{}
\maketitle

\begin{abstract}
We study generating sets for the kernels of the Frobenius polarization
operators associated with elements of finite groups.
By applying the discrete Fourier transform on each coset, we describe
these kernels in terms of additive semigroups defined by zero-sum relations
among roots of unity. We show that the monomials corresponding to the
minimal nonzero elements of these semigroups form generating sets of
minimum cardinality.
We further describe the minimal nonzero elements in the general case
in terms of those in the cyclic case and obtain a formula for the
minimum number of generators.
For the operator associated with a generator of a cyclic group, we also
show, as a consequence of known results on integer zero-sum relations
and circulant determinant representations, that its kernel and the
subalgebra generated by these circulant determinants have the same
field of fractions.
\end{abstract}

\section{Introduction}

Let $G$ be a finite group with identity element $e$.
For each $g\in G$, let $x_g$ be an indeterminate, with all these
indeterminates commuting and algebraically independent, and write
$\bm{x}:=\{x_g\mid g\in G\}$.
For each element $g\in G$, define the Frobenius polarization operators
$D_g$ and $\Delta_g$ by
\[
D_g:=\sum_{h\in G}x_{g^{-1}h}\frac{\partial}{\partial x_h},
\qquad
\Delta_g:=\sum_{h\in G}x_{hg^{-1}}\frac{\partial}{\partial x_h}.
\]
These operators were introduced by Frobenius~\cite[reprinted edition,
p.~50]{Frobenius1896} in his study of the group determinant $\Theta(G)$;
see also van der Waerden~\cite[p.~226]{MR803326} and
Yamaguchi--Yamaguchi~\cite[Sections~1 and~5]{YamaguchiYamaguchi2026}
for related expositions.
The Frobenius polarization operators satisfy
\[
D_g(x_h)=x_{g^{-1}h},\qquad \Delta_g(x_h)=x_{hg^{-1}}.
\]
Since they are derivations that send each variable to a homogeneous
linear polynomial, they are linear derivations.
It is known that the ring of constants of a linear derivation of a
polynomial ring over a field of characteristic zero is finitely generated.
Indeed, Nowicki--Tyc~\cite[Theorem~3.2]{MR1856934} state this result,
referring to Nagata's lectures on Hilbert's fourteenth
problem~\cite{MR215828} and Tyc's work on locally finite actions of
commutative Hopf algebras~\cite{MR1388860}.
Consequently, the kernels of the Frobenius polarization operators,
\begin{align*}
\Ker D_g&:=\{f(\bm{x})\in\mathbb C[\bm{x}]\mid D_gf(\bm{x})=0\},\\
\Ker\Delta_g&:=\{f(\bm{x})\in\mathbb C[\bm{x}]\mid\Delta_gf(\bm{x})=0\},
\end{align*}
are finitely generated $\mathbb C$-algebras.
Throughout this paper, generators and generating sets of a kernel
are understood to be those of a $\mathbb C$-algebra unless otherwise stated.

Nowicki--Strelcyn~\cite[Theorem~1]{NowickiStrelcyn1995} also show that,
even when the number of variables is fixed at an integer at least three,
the minimum number of generators of the ring of constants of a diagonal
derivation can be any nonnegative integer.

As for the common kernels of several Frobenius polarization operators,
it has been shown in~\cite{YamaguchiYamaguchi2026} that, for a finite
abelian group $G$,
\[
\bigcap_{g\in G\setminus\{e\}}\Ker D_g
=\bigcap_{g\in G\setminus\{e\}}\Ker\Delta_g
=\mathbb C[\Theta(G)].
\]
This equality has also been established for arbitrary finite groups,
without the assumption of commutativity~\cite{YamaguchiIto2026}.

For the linear derivation associated with a generator of a cyclic group
of order $n\geq3$, Moulin Ollagnier--Nowicki~\cite[Propositions~2.3
and~3.1]{MoulinOllagnierNowicki2013} describe a generating set of the
ring of constants in terms of the minimal nonzero exponent vectors with nonnegative integer entries
satisfying zero-sum relations among the $n$th roots
of unity, and characterize the minimum number of generators.
Moreover, writing $\omega(n)$ for the number of distinct prime divisors
of $n$, they show that the minimum number of generators equals
\[
\sum_{\substack{p\mid n\\p:\,\mathrm{prime}}}\frac{n}{p}
\]
if and only if $\omega(n)\leq2$~\cite[Corollary~3.13]{MoulinOllagnierNowicki2013}.
The same conclusion holds for $n=2$, as can be checked directly.

Yamaguchi--Ochiai--Yamaguchi~\cite[Section~2.1]{YamaguchiOchiaiYamaguchi2025}
describe the kernel using monomials in the Fourier variables whose
exponent vectors satisfy the zero-sum condition, and express the monomials
corresponding to prime cycles as circulant determinants~\cite[Lemma~13]{YamaguchiOchiaiYamaguchi2025}.
Combining Theorems~1 and~2 of that paper yields a generating set of the
kernel consisting of these circulant determinants, together with the
relations among the generators, when
$\omega(n)\leq2$~\cite[Theorems~1 and~2]{YamaguchiOchiaiYamaguchi2025}.
In contrast, when $\omega(n)\geq3$, these circulant determinants alone
do not generate the entire kernel as an
algebra~\cite[Theorem~3]{YamaguchiOchiaiYamaguchi2025}.

Starting from the known description of generating sets in the cyclic
case, we study generating sets for the kernel of the Frobenius
polarization operator $D_g$ associated with an element $g$ of an arbitrary
finite group $G$.
Writing $m=|\langle g\rangle|$ and $r=[G:\langle g\rangle]$, we describe
$\Ker D_g$ in terms of an additive semigroup $M_{m,r}$ defined by
zero-sum relations in which each $m$th root of unity occurs with
multiplicity $r$ (Theorem~\ref{thm:1}).
The monomials corresponding to its minimal nonzero elements form a
generating set of minimum cardinality (Theorem~\ref{thm:2}).
In particular, for arbitrary multiplicity $r$, we describe the minimal
nonzero elements of $M_{m,r}$ in terms of those of the semigroup $M_m$
associated with a cyclic group of order $m$, and give a counting formula
for the minimum number of generators in terms of $\Hilb(M_m)$ and $r$
(Theorem~\ref{thm:3}).

In the cyclic case, we also combine the R\'edei--de Bruijn--Schoenberg
theorem~\cite[Theorem~2.2]{MR1736695} with the circulant determinant
representation of Yamaguchi--Ochiai--Yamaguchi~\cite[Lemma~13]
{YamaguchiOchiaiYamaguchi2025} to make explicit, in our notation,
a consequence of known results: the kernel and the subalgebra generated
by these circulant determinants have the same field of fractions
(Theorem~\ref{thm:5}).
This consequence holds for every order $n\geq2$.
In particular, even when the circulant determinants do not generate
the entire kernel as an algebra, every monomial in the Fourier variables
that belongs to the kernel can be expressed as a ratio of products of
these circulant determinants.

For a finite group $G$ and an element $g\in G$, put
\[
H:=\langle g\rangle,\quad m:=|H|,\quad r:=[G:H],\quad
G=\bigsqcup_{i=1}^{r}Ht_i.
\]
Define
\begin{gather*}
\bm y:=\{y_{i,j}\mid 1\leq i\leq r,\ 0\leq j\leq m-1\},\\
y_{i,j}:=\sum_{k=0}^{m-1}\zeta_m^{jk}x_{g^kt_i},
\qquad \zeta_m:=\exp\!\left(\frac{2\pi\sqrt{-1}}{m}\right),
\end{gather*}
and let the additive semigroup $M_{m,r}$ and the $\mathbb C$-algebra
$\mathbb C[M_{m,r}]$ be given by
\begin{gather*}
M_{m,r}:=\left\{\bm\alpha=(\alpha_{i,j})\in\mathbb Z_{\geq0}^{|G|}
\ \middle|\ \sum_{i=1}^{r}\sum_{j=0}^{m-1}\alpha_{i,j}\zeta_m^j=0\right\},\\
\mathbb C[M_{m,r}]:=\mathbb C[\bm y^{\bm\alpha}\mid\bm\alpha\in M_{m,r}].
\end{gather*}

\begin{thm}\label{thm:1}
For the Frobenius polarization operator $D_g$, we have
\[
\Ker D_g
=\bigoplus_{\bm\alpha\in M_{m,r}}\mathbb C\bm y^{\bm\alpha}
=\mathbb C[M_{m,r}].
\]
\end{thm}

The structure of the additive semigroup $M_{m,r}$ depends only on $m$
and $r$. Thus Theorem~\ref{thm:1} gives the following corollary.

\begin{cor}\label{cor:2}
Let $G$ and $G'$ be finite groups, and let $g\in G$ and $g'\in G'$.
If
\[
|\langle g\rangle|=|\langle g'\rangle|,
\qquad [G:\langle g\rangle]=[G':\langle g'\rangle],
\]
then there is an isomorphism of $\mathbb C$-algebras
\[
\Ker D_g\cong\Ker D_{g'}.
\]
\end{cor}

Theorem~\ref{thm:1} and Dickson's lemma yield a generating set for
the kernel of a Frobenius polarization operator in terms of the minimal
elements of $M_{m,r}\setminus\{\bm0\}$.

\begin{thm}\label{thm:2}
Let $\bm\beta_1,\bm\beta_2,\ldots,\bm\beta_s$ be all the minimal elements
of $M_{m,r}\setminus\{\bm0\}$ with respect to the componentwise partial
order. Then
\[
\Ker D_g=\mathbb C[\bm y^{\bm\beta_1},\bm y^{\bm\beta_2},\ldots,
\bm y^{\bm\beta_s}].
\]
Moreover, these monomials form a minimal generating set of $\Ker D_g$,
and the minimum number of generators of $\Ker D_g$ is $s$.
When $g=e$, we have $\Ker D_e=\mathbb C$; in this case, the generating
set is taken to be empty and the minimum number of generators is zero.
\end{thm}

For the additive semigroup
\[
M_m:=\left\{\bm c=(c_0,c_1,\ldots,c_{m-1})\in\mathbb Z_{\geq0}^m
\ \middle|\ \sum_{j=0}^{m-1}c_j\zeta_m^j=0\right\},
\]
denote the set of minimal elements of $M_m\setminus\{\bm0\}$ by
\[
\Hilb(M_m).
\]
Define a map $\pi\colon M_{m,r}\to M_m$ by
\[
\pi(\bm\alpha)_j:=\sum_{i=1}^{r}\alpha_{i,j}.
\]

\begin{thm}\label{thm:3}
The set
\[
\left\{\bm y^{\bm\alpha}\mid\bm\alpha\in\pi^{-1}(\Hilb(M_m))\right\}
\]
is a minimal generating set of $\Ker D_g$. Its cardinality is
\[
\sum_{(c_0,c_1,\ldots,c_{m-1})\in\Hilb(M_m)}
\prod_{j=0}^{m-1}\binom{c_j+r-1}{r-1}.
\]
\end{thm}

We next give a result concerning the generators of $\Ker D_1$ when
$G=\mathbb Z/n\mathbb Z=\{0,1,\ldots,n-1\}$ is a cyclic group of
order $n\geq2$.
For each prime divisor $p$ of $n$, put $n_p:=n/p$.
For $0\leq i\leq n_p-1$, define
\[
y^{(p)}_{i,j}:=\sum_{l=0}^{n_p-1}\zeta_n^{i(j+pl)}x_{j+pl}
\qquad (0\leq j\leq p-1)
\]
and
\[
\Theta_p(\bm y_i^{(p)}):=
\det\begin{pmatrix}
y^{(p)}_{i,0}&y^{(p)}_{i,p-1}&\cdots&y^{(p)}_{i,1}\\
y^{(p)}_{i,1}&y^{(p)}_{i,0}&\cdots&y^{(p)}_{i,2}\\
\vdots&\vdots&\ddots&\vdots\\
y^{(p)}_{i,p-1}&y^{(p)}_{i,p-2}&\cdots&y^{(p)}_{i,0}
\end{pmatrix}.
\]
Let
\[
R_n:=\mathbb C[\Theta_p(\bm y_i^{(p)})\mid
p\text{ is a prime divisor of }n,\ 0\leq i\leq n_p-1]
\]
be the $\mathbb C$-subalgebra of $\mathbb C[\bm x]$ generated by these
circulant determinants.
For an integral domain $A$, write
\[
\Frac(A):=\left\{\frac ab\ \middle|\ a,b\in A,\ b\neq0\right\}
\]
for its field of fractions.

\begin{thm}\label{thm:5}
For the cyclic group $G=\mathbb Z/n\mathbb Z=\langle1\rangle$ of order
$n\geq2$, we have
\[
\Ker D_1\subset\Frac(\Ker D_1)=\Frac(R_n).
\]
\end{thm}

Let $\bm e_0,\bm e_1,\ldots,\bm e_{n-1}$ be the standard basis of
$\mathbb Z^n$. For each prime divisor $p$ of $n$, put
\[
\bm v_i^{(p)}:=\sum_{j=0}^{p-1}\bm e_{i+n_pj}
\qquad(0\leq i\leq n_p-1).
\]
The proof of Theorem~\ref{thm:5} yields the following corollary.

\begin{cor}\label{cor:6}
For every $\bm\alpha\in M_{n,1}=M_n$, the monomial
$\bm y^{\bm\alpha}\in\Ker D_1$ can be expressed as a ratio of products
of the circulant determinants $\Theta_p(\bm y_i^{(p)})$.
More precisely, there exist integers $c_{p,i}$ such that
\[
\bm\alpha=\sum_{\substack{p\mid n\\p:\,\text{prime}}}
\sum_{i=0}^{n_p-1}c_{p,i}\bm v_i^{(p)}.
\]
Writing $c_{p,i}=c_{p,i}^+-c_{p,i}^-$ with
$c_{p,i}^+,c_{p,i}^-\in\mathbb Z_{\geq0}$, we obtain
\[
\bm y^{\bm\alpha}
=\prod_{\substack{p\mid n\\p:\,\text{prime}}}
\prod_{i=0}^{n_p-1}
\frac{\Theta_p(\bm y_i^{(p)})^{c_{p,i}^+}}
{\Theta_p(\bm y_i^{(p)})^{c_{p,i}^-}}.
\]
\end{cor}

Moreover, analogues of the results from Theorem~\ref{thm:1}
to Corollary~\ref{cor:6} hold for the derivation $\Delta_g$~(see Section~\ref{sec:6}).

\section{Proof of Theorem~\ref{thm:1}}

We apply the Fourier diagonalization for the cyclic
case~\cite[Propositions~2.1 and~2.2]{MoulinOllagnierNowicki2013},
\cite[Section~2.1]{YamaguchiOchiaiYamaguchi2025}
to the variables associated with each right coset $Ht_i$.
For a finite group $G$ and an element $g\in G$, put
\[
H:=\langle g\rangle,\quad m:=|H|,\quad r:=[G:H],\quad
G=\bigsqcup_{i=1}^{r}Ht_i.
\]
We replace the coordinates $\bm x$ by new coordinates $\bm y$ using
the discrete Fourier transform:
\begin{gather*}
\bm y:=\{y_{i,j}\mid1\leq i\leq r,\ 0\leq j\leq m-1\},\\
y_{i,j}:=\sum_{k=0}^{m-1}\zeta_m^{jk}x_{g^kt_i},
\qquad\zeta_m:=\exp\!\left(\frac{2\pi\sqrt{-1}}{m}\right).
\end{gather*}
The inverse transform is
\[
x_{g^kt_i}=\frac1m\sum_{j=0}^{m-1}\zeta_m^{-jk}y_{i,j}
\qquad(1\leq i\leq r,\ 0\leq k\leq m-1),
\]
so $\mathbb C[\bm x]=\mathbb C[\bm y]$.
We then have the following proposition.

\begin{prop}\label{prop:2.1}
For each variable $y_{i,j}$, we have
\[
D_g(y_{i,j})=\zeta_m^jy_{i,j}.
\]
Hence, in the new coordinates, the Frobenius polarization operator is
\[
D_g=\sum_{i=1}^{r}\sum_{j=0}^{m-1}
\zeta_m^jy_{i,j}\frac{\partial}{\partial y_{i,j}}.
\]
\end{prop}
\begin{proof}
By the definition of the Frobenius polarization operator,
\begin{align*}
D_g(y_{i,j})
&=\sum_{k=0}^{m-1}\zeta_m^{jk}x_{g^{k-1}t_i}\\
&=\zeta_m^j\sum_{k=0}^{m-1}\zeta_m^{jk}x_{g^kt_i}
=\zeta_m^jy_{i,j}.\qedhere
\end{align*}
\end{proof}

Define the additive semigroup $M_{m,r}$ and the $\mathbb C$-algebra
$\mathbb C[M_{m,r}]$ by
\begin{gather*}
M_{m,r}:=\left\{\bm\alpha=(\alpha_{i,j})\in\mathbb Z_{\geq0}^{|G|}
\ \middle|\ \sum_{i=1}^{r}\sum_{j=0}^{m-1}\alpha_{i,j}\zeta_m^j=0\right\},\\
\mathbb C[M_{m,r}]:=\mathbb C[\bm y^{\bm\alpha}\mid\bm\alpha\in M_{m,r}].
\end{gather*}

\begin{proof}[Proof of Theorem~\ref{thm:1}]
For an exponent vector $\bm\alpha=(\alpha_{i,j})\in
\mathbb Z_{\geq0}^{rm}=\mathbb Z_{\geq0}^{|G|}$, define
\[
\bm y^{\bm\alpha}:=\prod_{i=1}^{r}\prod_{j=0}^{m-1}y_{i,j}^{\alpha_{i,j}}.
\]
Proposition~\ref{prop:2.1} and the Leibniz rule give
\[
D_g(\bm y^{\bm\alpha})
=\left(\sum_{i=1}^{r}\sum_{j=0}^{m-1}\alpha_{i,j}\zeta_m^j\right)
\bm y^{\bm\alpha}.
\]
Thus, for a monomial $\bm y^{\bm\alpha}$,
\[
\bm y^{\bm\alpha}\in\Ker D_g
\iff D_g\bm y^{\bm\alpha}=0
\iff\bm\alpha\in M_{m,r}.
\]
Distinct monomials are linearly independent over $\mathbb C$, and
$D_g$ maps each monomial to a scalar multiple of itself. Therefore,
\[
\Ker D_g
=\bigoplus_{\bm\alpha\in M_{m,r}}\mathbb C\bm y^{\bm\alpha}
=\mathbb C[M_{m,r}].\qedhere
\]
\end{proof}

\section{Proof of Theorem~\ref{thm:2}}

We use Dickson's lemma to obtain a generating set of $M_{m,r}$.

\begin{lem}[Dickson's lemma~\cite{MR1506194,MR4952933,MR3838370}]
Equip $\mathbb Z_{\geq0}^n$ with the partial order $\leq$ defined by
\[
(a_1,a_2,\ldots,a_n)\leq(b_1,b_2,\ldots,b_n)
\stackrel{\mathrm{def.}}{\iff}
a_i\leq b_i\quad(i=1,2,\ldots,n).
\]
Then every subset $S\subset\mathbb Z_{\geq0}^n$ has only finitely many
minimal elements with respect to this partial order.
\end{lem}

\begin{lem}\label{lem:3.2}
Let $\lambda_1,\lambda_2,\ldots,\lambda_n$ be elements of an abelian
group $\Lambda$, and consider the additive semigroup
\[
M:=\{\bm a=(a_1,a_2,\ldots,a_n)\in\mathbb Z_{\geq0}^n
\mid a_1\lambda_1+a_2\lambda_2+\cdots+a_n\lambda_n=0\}.
\]
If $\bm h_1,\bm h_2,\ldots,\bm h_s$ are all the minimal elements
of $M\setminus\{\bm0\}$, then
\[
M=\mathbb Z_{\geq0}\bm h_1+\mathbb Z_{\geq0}\bm h_2+
\cdots+\mathbb Z_{\geq0}\bm h_s.
\]
\end{lem}
\begin{proof}
By Dickson's lemma, the set $M\setminus\{\bm0\}$ has finitely many
minimal elements; denote them by $\bm h_1,\bm h_2,\ldots,\bm h_s$.
For any $\bm a\in M\setminus\{\bm0\}$, there exists a minimal element
$\bm h_i$ with $\bm h_i\leq\bm a$. Then $\bm a-\bm h_i\in M$.
If $\bm a-\bm h_i=\bm0$, then $\bm a=\bm h_i$, and hence $\bm a$
is a nonnegative integer linear combination of the minimal elements.
If $\bm a-\bm h_i\in M\setminus\{\bm0\}$, there exists a minimal
element $\bm h_j$ with $\bm h_j\leq\bm a-\bm h_i$.
Then $\bm a-\bm h_i-\bm h_j\in M$.
At each step, the sum of the coordinates of the remaining vector
strictly decreases, so the process reaches $\bm0$ after finitely many
steps. Thus $\bm a$ is a nonnegative integer linear combination of the
minimal elements.
\end{proof}

\begin{lem}
For a nonzero element $\bm h$ of the additive semigroup $M$ in
Lemma~\ref{lem:3.2}, being a minimal element of $M\setminus\{\bm0\}$
is equivalent to being an irreducible element of $M$.
\end{lem}
\begin{proof}
Suppose that a minimal element $\bm h$ is not irreducible, so that
$\bm h=\bm a+\bm b$ for some $\bm a,\bm b\in M\setminus\{\bm0\}$.
Then $\bm0<\bm a<\bm h$, contradicting the minimality of $\bm h$.
Thus every minimal element is irreducible.
Conversely, if $\bm h$ is not a minimal element of $M\setminus\{\bm0\}$,
then there exists $\bm b\in M\setminus\{\bm0\}$ with $\bm b<\bm h$.
In this case, $\bm h-\bm b\in M\setminus\{\bm0\}$, and the decomposition
$\bm h=\bm b+(\bm h-\bm b)$ shows that $\bm h$ is not irreducible.
\end{proof}

For generating sets and the minimum number of generators in the case
$r=1$, see Moulin Ollagnier--Nowicki~\cite[Propositions~2.3 and~3.1]
{MoulinOllagnierNowicki2013}.
For arbitrary $r$, we prove below that the monomials corresponding to
the minimal nonzero exponent vectors generate the kernel, and use the
quotient space $I/I^2$ to establish the minimality of their number.

\begin{proof}[Proof of Theorem~\ref{thm:2}]
Applying Lemma~\ref{lem:3.2} with the abelian group
$\Lambda=(\mathbb C,+)$ and $\lambda_{i,j}=\zeta_m^j$, we obtain
minimal elements $\bm\beta_1,\bm\beta_2,\ldots,\bm\beta_s$ of
$M_{m,r}\setminus\{\bm0\}$ such that
\[
M_{m,r}=\mathbb Z_{\geq0}\bm\beta_1+\mathbb Z_{\geq0}\bm\beta_2+
\cdots+\mathbb Z_{\geq0}\bm\beta_s.
\]
Hence Theorem~\ref{thm:1} gives
\begin{align*}
\Ker D_g
&=\bigoplus_{\bm\alpha\in M_{m,r}}\mathbb C\bm y^{\bm\alpha}\\
&=\mathbb C[\bm y^{\bm\beta_1},\bm y^{\bm\beta_2},\ldots,
\bm y^{\bm\beta_s}].
\end{align*}
Next, we prove that this generating set is minimal and that its
cardinality $s$ is the minimum possible.
Put $A:=\Ker D_g=\mathbb C[M_{m,r}]$, and define an ideal $I$ of $A$ by
\[
I:=\bigoplus_{\bm\alpha\in M_{m,r}\setminus\{\bm0\}}
\mathbb C\bm y^{\bm\alpha}.
\]
A nonzero element of the additive semigroup $M_{m,r}$ is irreducible
if and only if it is a minimal element of $M_{m,r}\setminus\{\bm0\}$.
Consequently, the ideal $I^2$ is spanned as a $\mathbb C$-vector space
by the monomials whose exponent vectors can be written as sums of two
nonzero elements of $M_{m,r}$.
It follows that the quotient space $I/I^2$ has the $\mathbb C$-basis
\[
\bm y^{\bm\beta_1}+I^2,\bm y^{\bm\beta_2}+I^2,\ldots,
\bm y^{\bm\beta_s}+I^2,
\]
and hence
\[
\dim_{\mathbb C}I/I^2=s.
\]
Now suppose that $A=\mathbb C[f_1,\ldots,f_t]$.
Subtracting the constant term from each $f_i$, we may assume that
$f_i\in I$.
Since every product of two or more of the $f_i$ is zero in $I/I^2$,
the residue classes
\[
f_1+I^2,f_2+I^2,\ldots,f_t+I^2
\]
span $I/I^2$ as a $\mathbb C$-vector space. Thus
\[
t\geq\dim_{\mathbb C}I/I^2=s.
\]
Therefore, the $s$ monomials above form a minimal generating set,
and every generating set has at least $s$ elements.
\end{proof}

\section{Proof of Theorem~\ref{thm:3}}

\begin{lem}\label{lem:4.1}
The set of minimal elements of $M_{m,r}\setminus\{\bm0\}$ is
\[
\pi^{-1}(\Hilb(M_m)).
\]
\end{lem}
\begin{proof}
We first show by contradiction that every minimal element of
$M_{m,r}\setminus\{\bm0\}$ belongs to $\pi^{-1}(\Hilb(M_m))$.
Let $\bm\alpha$ be a minimal element of $M_{m,r}\setminus\{\bm0\}$,
and suppose that $\bm c:=\pi(\bm\alpha)$ is not a minimal element of
$M_m\setminus\{\bm0\}$.
Then there exists $\bm d\in M_m$ such that $\bm0\neq\bm d<\bm c$.
For each $j$, we have $d_j\leq c_j=\sum_{i=1}^{r}\alpha_{i,j}$, so
we can choose nonnegative integers $\gamma_{i,j}$ satisfying
\[
0\leq\gamma_{i,j}\leq\alpha_{i,j},\qquad
\sum_{i=1}^{r}\gamma_{i,j}=d_j.
\]
The vector $\bm\gamma=(\gamma_{i,j})$ then satisfies
$\bm0\neq\bm\gamma<\bm\alpha$ and
\[
\sum_{j=0}^{m-1}\sum_{i=1}^{r}\gamma_{i,j}\zeta_m^j
=\sum_{j=0}^{m-1}d_j\zeta_m^j=0.
\]
Thus $\bm\gamma\in M_{m,r}$, contradicting the minimality of
$\bm\alpha$. Hence $\bm c\in\Hilb(M_m)$.

Conversely, we show by contradiction that every element of
$\pi^{-1}(\Hilb(M_m))$ is a minimal element of
$M_{m,r}\setminus\{\bm0\}$.
Let $\bm c\in\Hilb(M_m)$ and $\bm\alpha\in\pi^{-1}(\bm c)$.
Suppose that there exists $\bm\gamma\in M_{m,r}$ such that
$\bm0\neq\bm\gamma<\bm\alpha$.
Then $\bm0\neq\pi(\bm\gamma)<\bm c$, contradicting the minimality
of $\bm c$.
Therefore, the vector $\bm\alpha$ is a minimal element of
$M_{m,r}\setminus\{\bm0\}$.
\end{proof}

\begin{proof}[Proof of Theorem~\ref{thm:3}]
The first assertion follows from Theorem~\ref{thm:2} and
Lemma~\ref{lem:4.1}.
Fix an element $\bm c=(c_0,c_1,\ldots,c_{m-1})\in\Hilb(M_m)$.
The condition $\pi(\bm\alpha)=\bm c$ for $\bm\alpha=(\alpha_{i,j})$
is equivalent to
\[
\alpha_{1,j}+\alpha_{2,j}+\cdots+\alpha_{r,j}=c_j,
\qquad\alpha_{i,j}\in\mathbb Z_{\geq0}
\]
for each $j$. The number of nonnegative integer solutions is
\[
\binom{c_j+r-1}{r-1}.
\]
Since the choices for distinct values of $j$ are independent, we obtain
\[
|\pi^{-1}(\bm c)|=\prod_{j=0}^{m-1}\binom{c_j+r-1}{r-1}.
\]
Summing over $\bm c\in\Hilb(M_m)$ proves the assertion.
\end{proof}

\section{Proof of Theorem~\ref{thm:5}}
\label{sec:cyclic}

In this section, we prove Theorem~\ref{thm:5} as a consequence of known
results. It follows by combining the results on fields of fractions
and integer zero-sum relations in
Moulin Ollagnier--Nowicki~\cite[Propositions~2.4 and~2.5,
Theorem~3.8]{MoulinOllagnierNowicki2013} with the circulant determinant
representation of Yamaguchi--Ochiai--Yamaguchi~\cite[Lemma~13]
{YamaguchiOchiaiYamaguchi2025}.

We use the group-ring formulation of the R\'edei--de Bruijn--Schoenberg
theorem~\cite[Theorem~2.2]{MR1736695} to show that the exponent vectors
of the monomials generating the kernel and those of the monomials
corresponding to the circulant determinants generate the same additive
group. From this, we directly deduce that the field of fractions of the
kernel coincides with that of the subalgebra generated by these circulant
determinants.

For each prime divisor $p$ of $n$, put $n_p:=n/p$.
The following result is due to R\'edei~\cite[Hilfssatz~4]{MR45729}
\cite{MR62760}, de Bruijn~\cite[Theorem~1]{MR59271}, and
Schoenberg~\cite[Theorem~1]{MR170887}.
Lam--Leung~\cite[Theorem~2.2]{MR1736695} give the following
formulation in terms of group rings.

\begin{thm}[R\'edei--de Bruijn--Schoenberg]\label{thm:13}
Let $G=\langle g\rangle$ be a cyclic group of order $n$, with prime
factorization $n=p_1^{a_1}\cdots p_s^{a_s}$.
Define a ring homomorphism from the group ring $\mathbb ZG$ to
$\mathbb Z[\zeta_n]$ by
\[
\phi(g):=\zeta_n.
\]
For each prime divisor $p$ of $n$, let
\[
P_p:=\langle g^{n_p}\rangle
\]
be the subgroup of $G$ of order $p$, and put
\[
\sigma(P_p):=\sum_{h\in P_p}h
=1+g^{n_p}+\cdots+g^{(p-1)n_p}.
\]
Then
\[
\Ker\phi=\sum_{\substack{p\mid n\\p:\,\text{prime}}}
\mathbb ZG\,\sigma(P_p).
\]
\end{thm}

Let $\bm e_0,\bm e_1,\ldots,\bm e_{n-1}$ be the standard basis of
$\mathbb Z^n$.
For each prime divisor $p$ of $n$, put
\[
\bm v_i^{(p)}:=\sum_{j=0}^{p-1}\bm e_{i+n_pj}
\qquad(0\leq i\leq n_p-1),
\]
and define
\[
L_n:=\left\{\bm u=(u_0,u_1,\ldots,u_{n-1})\in\mathbb Z^n
\ \middle|\ \sum_{j=0}^{n-1}u_j\zeta_n^j=0\right\}.
\]
Rewriting Theorem~\ref{thm:13} in terms of exponent vectors yields
the following corollary.

\begin{cor}\label{cor:14}
For every integer $n\geq2$, we have
\[
L_n=\sum_{\substack{p\mid n\\p:\,\mathrm{prime}}}
\sum_{i=0}^{n_p-1}\mathbb Z\,\bm v_i^{(p)}.
\]
In other words, all integer relations among the $n$th roots of unity
are generated, as a $\mathbb Z$-module, by cyclic shifts of the prime
cycles corresponding to the prime divisors of $n$.
\end{cor}

Define the set
\[
T_n:=\bigcup_{\substack{p\mid n\\p:\,\mathrm{prime}}}
\left\{\bm v_i^{(p)}\ \middle|\ 0\leq i\leq n_p-1\right\}.
\]
Write $\Gp(M_n)$ for the group generated by the additive semigroup
$M_n$:
\[
\Gp(M_n):=\{\bm\alpha-\bm\beta\mid\bm\alpha,\bm\beta\in M_n\}.
\]

The equality $\Gp(M_n)=L_n$ used in the argument for the following
lemma corresponds to Moulin Ollagnier--Nowicki~\cite[Proposition~1.3]
{MoulinOllagnierNowicki2013}.
The fact that $L_n$ is generated by $T_n$ as a $\mathbb Z$-module
follows from Corollary~\ref{cor:14}; see also
\cite[Theorem~3.8]{MoulinOllagnierNowicki2013}.

\begin{lem}\label{lem:15}
For the set $T_n$ and the additive semigroup $M_n$, we have
\[
\mathbb Z\mathchar`-\Span T_n=\Gp(M_n).
\]
\end{lem}
\begin{proof}
Corollary~\ref{cor:14} gives
\[
L_n=\mathbb Z\mathchar`-\Span T_n.
\]
Since $M_n\subset L_n$ and $L_n$ is an additive group, we have
\[
\Gp(M_n)\subset L_n.
\]
To prove the reverse inclusion, take any
\[
\bm u=(u_0,u_1,\ldots,u_{n-1})\in L_n.
\]
Since $\bm1:=(1,1,\ldots,1)\in M_n$, we may choose a sufficiently
large nonnegative integer $N$ such that
$\bm u+N\bm1\in\mathbb Z_{\geq0}^n$. Then
\[
\bm u+N\bm1\in M_n,\qquad N\bm1\in M_n,
\]
and
\[
\bm u=(\bm u+N\bm1)-N\bm1.
\]
Thus
\[
L_n\subset\Gp(M_n).
\]
Combining the inclusions, we obtain
\[
\mathbb Z\mathchar`-\Span T_n=\Gp(M_n)=L_n.\qedhere
\]
\end{proof}

For each prime divisor $p$ of $n$, put $n_p:=n/p$.
For $0\leq i\leq n_p-1$, define
\[
y^{(p)}_{i,j}:=\sum_{l=0}^{n_p-1}\zeta_n^{i(j+pl)}x_{j+pl}
\qquad(0\leq j\leq p-1)
\]
and
\[
\Theta_p(\bm y_i^{(p)}):=
\det\begin{pmatrix}
y^{(p)}_{i,0}&y^{(p)}_{i,p-1}&\cdots&y^{(p)}_{i,1}\\
y^{(p)}_{i,1}&y^{(p)}_{i,0}&\cdots&y^{(p)}_{i,2}\\
\vdots&\vdots&\ddots&\vdots\\
y^{(p)}_{i,p-1}&y^{(p)}_{i,p-2}&\cdots&y^{(p)}_{i,0}
\end{pmatrix}.
\]
Let
\[
R_n:=\mathbb C[\Theta_p(\bm y_i^{(p)})\mid
p\text{ is a prime divisor of }n,\ 0\leq i\leq n_p-1]
\]
be the $\mathbb C$-subalgebra of $\mathbb C[\bm x]$ generated by these
determinants.
Here $G=\mathbb Z/n\mathbb Z$ and $g=1$, so $H=G$, $m=n$, and $r=1$.
Choose $t_1=0$ as the coset representative and abbreviate the Fourier
variables $y_{1,k}$ as
\[
y_k:=y_{1,k}=\sum_{l=0}^{n-1}\zeta_n^{kl}x_l
\qquad(0\leq k\leq n-1).
\]
Writing $\bm y:=(y_0,y_1,\ldots,y_{n-1})$, we have
\[
\bm y^{\bm v_i^{(p)}}=\prod_{j=0}^{p-1}y_{i+n_pj}.
\]
By \cite[Lemma~13]{YamaguchiOchiaiYamaguchi2025},
\[
\Theta_p(\bm y_i^{(p)})=\bm y^{\bm v_i^{(p)}}
=\prod_{j=0}^{p-1}y_{i+n_pj}.
\]
It follows that
\[
R_n=\mathbb C\left[\bm y^{\bm v}\ \middle|\ \bm v\in T_n\right].
\]

\begin{proof}[Proof of Theorem~\ref{thm:5}]
By Theorem~\ref{thm:1},
\[
\Ker D_1=\mathbb C[M_n].
\]
On the other hand, the algebra $R_n$ is generated by the monomials
corresponding to $T_n$.
In general, the field of fractions of a $\mathbb C$-subalgebra generated
by monomials is generated by those monomials as a field over $\mathbb C$.
Since each generating monomial is invertible in this field, the field
of fractions is determined solely by the additive group generated by
the exponent vectors of the generators.
Lemma~\ref{lem:15} gives
\[
\Gp(M_n)=\mathbb Z\mathchar`-\Span(T_n).
\]
Thus the additive group generated by the exponent vectors corresponding
to $\Ker D_1$ coincides with that generated by the exponent vectors
of the generators of $R_n$. Therefore,
\[
\Frac(\Ker D_1)=\Frac(R_n).\qedhere
\]
\end{proof}

\section{The kernel of the derivation $\Delta_g$}\label{sec:6}

For a finite group $G$ and an element $g\in G$,
set $H:=\langle g\rangle$, $m:=|H|$, and $r:=[G:H]$.
Given a left coset decomposition $G=\bigsqcup_{i=1}^{r}s_iH$,
define
\[
z_{i,j}:=\sum_{k=0}^{m-1}\zeta_m^{jk}x_{s_i g^k}
\qquad (1\leq i\leq r,\ 0\leq j\leq m-1).
\]
Then, as in the case of the operator $D_g$, we have
\[
\Delta_g(z_{i,j})=\zeta_m^j z_{i,j}.
\]
Thus, replacing the operator $D_g$ and the variables $y_{i,j}$
by $\Delta_g$ and $z_{i,j}$, respectively, yields
analogues of Theorems~\ref{thm:1},~\ref{thm:2}, and~\ref{thm:3}
and Corollary~\ref{cor:2}.
Moreover, when the group $G$ is cyclic, we have $D_g=\Delta_g$,
so Theorem~\ref{thm:5} and its corollary also hold without modification.

\clearpage
\bibliography{reference}

\medskip
\begin{flushleft}
Faculty of Education,\\
University of Miyazaki,\\
1-1 Gakuen Kibanadai-nishi,\\
Miyazaki 889-2192, Japan\\
\textit{Email address}: \href{mailto:n-yamaguchi@miyazaki-u.ac.jp}{n-yamaguchi@miyazaki-u.ac.jp}
\end{flushleft}

\end{document}